\documentclass{article}
\usepackage[utf8]{inputenc}

\usepackage{amsmath}
\usepackage{amsthm}
\usepackage{hyperref}
\usepackage{relsize}
\usepackage[capitalise, noabbrev]{cleveref}
\usepackage[hideerrors]{xcolor}
\usepackage{pgf,tikz}
\usepackage{bold-extra}
\usepackage{multirow}
\usepackage{diagbox}
\usepackage{CountingDomineering}

\usepackage{array}
\newcolumntype{L}[1]{>{\raggedright\let\newline\\\arraybackslash\hspace{0pt}}m{#1}}
\newcolumntype{C}[1]{>{\centering\let\newline\\\arraybackslash\hspace{0pt}}m{#1}}
\newcolumntype{R}[1]{>{\raggedleft\let\newline\\\arraybackslash\hspace{0pt}}m{#1}}

\newcommand{\zero}{\mathbf{0}}

\newcommand{\zeroblock}{
	\parbox[c][1cm]{1cm}{
		\makebox[10mm][c]{$\mathbf{0}$}
	}
}

\theoremstyle{plain}
\newtheorem{theorem}{Theorem}

\newtheorem{proposition}[theorem]{Proposition}
\theoremstyle{definition}

\newtheorem{example}[theorem]{Example}

\theoremstyle{remark}
\newtheorem{remark}[theorem]{Remark}

\newcommand{\domineering}{\textsc{Domineering }}
\newcommand{\Domineering}{\textsc{Domineering}}

\usepackage{arydshln}

\newcommand{\seqnum}[1]{\href{https://oeis.org/#1}{\underline{#1}}}

\begin{document}

\begin{center}
\vskip 1cm{\LARGE\bf Counting \domineering Positions}
\vskip 1cm
Svenja Huntemann \\
School of Mathematics and Statistics\\
Carleton University\\
Ottawa, ON\\
Canada\\
\href{mailto:svenja.huntemann@carleton.ca}{\tt svenja.huntemann@carleton.ca} \\
\ \\
Neil Anderson McKay \\
Department of Mathematics and Statistics \\
University of New Brunswick \\
Saint John, NB \\
Canada \\

\href{mailto:neil.mckay@unb.ca}{\tt neil.mckay@unb.ca} \\
\end{center}

\vskip .2 in

\begin{abstract}
\domineering is a two player game played on a checkerboard in which one player places dominoes vertically and the other places them horizontally. We give bivariate generating polynomials enumerating \domineering positions by the number of each player's pieces. 
We enumerate all positions, maximal positions, and positions where one player has no move. Using these polynomials we count the number of positions that occur during alternating play. Our method extends to enumerating positions from mid-game positions and we include an analysis of a tournament game.
\end{abstract}

\section{Introduction}

Combinatorial games are 2-player games with perfect information and no chance devices, such as \textsc{Chess} or \textsc{Go}. Many combinatorial games have, for a fixed starting position, a finite number of options and the game is guaranteed to end in a finite number of moves; in theory we could determine by computer which player would win if both players play perfectly. In practice game theorists and computer scientists have not determined the outcomes of games under perfect-play because of the complexity of the required search.

Enumeration of positions has been studied, directly or indirectly, for several combinatorial games. Papers on counting game positions consider the problem of enumerating specific types of positions --- \textsc{Go} end positions \cite{FarrSchmidt2008,Farr2003,TrompFarneback2007} and second-player win positions for specific lesser known games \cite{Hetyei2009,NowakowskiRLMM2013}. In the game \textsc{Node Kayles}, played on a graph, the two players alternate choosing vertices not adjacent to any previously chosen ones, thus forming an independent set. Therefore the independence polynomial of a graph is equal to the generating polynomial for positions of \textsc{Node Kayles} on that graph. Similarly, in the play of \textsc{Arc Kayles} the players form a matching. The enumeration of matchings has been studied for many graphs (see \cref{sec:relatedProblems}). 
In \emph{partizan} games, where players may have distinct options at some point during play, the convention is to call the two players \textbf{Left} (who uses bLue pieces) and \textbf{Right} (who uses Red pieces). In \textsc{Col}, which is played on a graph, a move for Left is to color a vertex blue, while a move for Right is to color a vertex red, and no two vertices of the same color may be adjacent. Oh and Lee \cite{OhLee2016} call such a position a bipartite independent vertex set and give the generating polynomial for grid graphs. Brown et al.\ \cite{BrownCHMMNS2019} give the generating polynomial, which they call the \textbf{polynomial profile}, for several games, including closed forms for \textsc{Col} and the game \textsc{Snort} (like \textsc{Col}, but two vertices of \emph{different} colors may not be adjacent) played on paths.

In this paper we consider the game \textbf{\Domineering}. This game is played on a checkerboard. The two players alternately place dominoes on adjacent empty squares; Left places vertically and Right places horizontally. The game ends when the player whose turn it is cannot place a piece; the player who cannot play loses --- this is the \textbf{normal play} convention. We will count all \domineering positions, as well as \domineering positions with certain properties, following the method used by Oh and Lee in \cite{Oh2017,OhLee2016}. In newer work Oh also considers monomer-dimer tilings in \cite{Oh2019}, thus implicitly counting \domineering positions, which overlaps with some of our work as discussed in \cref{sec:general}.

\Domineering, also called \textsc{Dominoes} or \textsc{Crosscram} in older work, was introduced by G\"oran Andersson and popularized by Martin Gardner in the 1970s. There has been significant interest in this game since then (see for example \cite{Berlekamp1988,BerlekampCG2004,Uiterwijk2016,West1996,Wolfe1993} and references therein). Some simplification techniques are known and can be combined with computer search for analysis, but no complete solution is available. It is known who wins on many rectangular boards up to size $11\times 11$ and some non-rectangular boards, as well as some values (essentially who has the advantage and by how much) and temperatures (essentially the urgency of moving). For the reader further interested in combinatorial game theory techniques we recommend Siegel's Combinatorial Game Theory \cite{Siegel2013}.

As a game of \domineering progresses, the board often naturally breaks into smaller \textbf{components}, see \cref{fig:position-to-sum}, where play in one component cannot affect the available moves in other components. A player on their turn has to decide which component they want to play in and make their move there. A position (board) is the \textbf{disjunctive sum} of its constituent components.

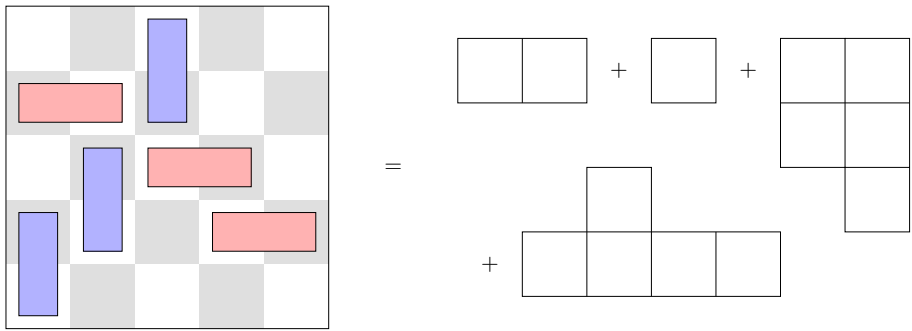
\begin{figure}[!ht]
\centering
\resizebox{\linewidth}{!}{
\begin{tikzpicture}
\tikzchessboard[4]
\Ldomino{1}{1}
\Ldomino{2}{3}
\Rdomino{3}{1}
\Ldomino{0}{0}
\Rdomino{0}{3}
\Rdomino{2}{2}
\node at (6,2.5) {$=$};
\draw (7,3.5)--(9,3.5)--(9,4.5)--(7,4.5)--cycle;
\draw (8,3.5)--(8,4.5);
\node at (9.5,4) {$+$};
\draw (10,3.5)--(11,3.5)--(11,4.5)--(10,4.5)--cycle;
\node at (11.5,4) {$+$};
\draw (14,2.5)--(12,2.5)--(12,4.5)--(13,4.5)--(13,1.5)--(14,1.5)--(14,4.5)--(13,4.5);
\draw (12,3.5)--(14,3.5);
\node at (7.5,1) {$+$};
\draw (8,0.5)--(12,0.5)--(12,1.5)--(8,1.5)--cycle;
\draw (9,0.5)--(9,2.5)--(10,2.5)--(10,0.5);
\draw (11,0.5)--(11,1.5);
\end{tikzpicture}}
\caption{During play a \domineering position may decompose into a disjunctive sum of \domineering positions that are not necessarily rectangular. The board on the left breaks into the disjunctive sum of the components on the right.}
\label{fig:position-to-sum}
\end{figure}

If the two players play in different components in the disjunctive sum a player may make consecutive moves in the same component. For example in \cref{fig:position-to-sum}, if it is Right's turn, he could choose to play in the third component, Left may choose the fourth, and Right choose the third again. Due to this, we are also interested in \domineering positions in which the two players do not necessarily alternate turns, hence we consider plays where the difference between the number of Left dominoes and Right dominoes may be larger than 1.

In \cref{sec:general} we will find the generating function for the number of positions on an $m\times n$ rectangular board which we denote by 
\[
D_{m,n} (x,y) = \sum d(a,b) x^a y^b
\]
where $d(a,b)$ is the number of positions with $a$ Left dominoes and $b$ Right dominoes. We then demonstrate in \cref{sec:nonrectangular} how to generalize the technique employed to find $D_{m,n}(x,y)$ to calculate the generating function for non-rectangular boards.

We are also particularly interested in enumerating positions at the end of the game. A \textbf{Left end} is a position in which Left can no longer play a piece, while Right potentially still has moves; a \textbf{Right end} is defined similarly. A \textbf{maximal position} is a position which is both a Left end and a Right end --- a position in which no player can place a domino. In \cref{sec:maximal} we will find the generating function for maximal positions, which we denote by $F_{m,n}(x,y)$. In \cref{sec:ends} we consider the generating function for Right ends; this is essentially the same as the generating function for Left ends.

Finally, in \cref{sec:FurtherWork} we will discuss some other questions related to the technique used within and the problem of enumerating \domineering positions.

\subsection{Related problems}\label{sec:relatedProblems}

The checkerboard corresponds to a grid graph by using a vertex for each square and connecting them with an edge if the two squares are horizontally or vertically adjacent. Playing \domineering on the checkerboard is then equivalent to forming a matching (also called an edge independent set) on the grid graph, with a distinction made between horizontal and vertical edges. 

If we make no distinction between Left and Right dominoes, that is, we enumerate the positions with a fixed number of dominoes, this is equivalent to enumerating the number of matchings in a grid graph or enumerating the monomer-dimer tilings of the chessboard (Propp \cite{Propp1999} gives a summary of connections to physics and chemistry). In particular, this means that $D_{m,n}(x,x)$ is the generating function for matchings in an $m\times n$ grid graph, and that $D_{m,n}(1,1)$ gives the total number of such matchings. The latter, also known as the Hosoya index of the grid graph, is known (for example see Ahrens \cite{Ahrens1981} or the OEIS \cite{oeis}
sequences \seqnum{A030186} for $2\times n$, \seqnum{A033506} for $3\times n$, and \seqnum{A028420} for $n\times n$). If a perfect matching exists in the $m\times n$ grid graph, then the number of perfect matchings is the leading coefficient of $D_{m,n}(x,x)$ as this is the number of maximum matchings.

Similarly, $F_{m,n}(x,x)$ gives the generating function for the maximal matchings in an $m\times n$ grid graph. The total number of maximal positions $F_{m,n}(1,1)$ are known for some values of $m$ and $n$ (OEIS sequences \seqnum{A000931} for $1\times n$, \seqnum{A286945} for $2\times n$, \seqnum{A288028} for $3\times n$, and \seqnum{A287595} for $n\times n$). 

\domineering belongs to the class of strong placement games, and thus one can assign to each game a simplicial complex representing the legal positions (see Faridi et al.\ \cite{FaridiHN2019a} for details). The coefficients of $D_{m,n}(x,x)$ are then the entries of the $f$-vector --- the vector counting the number of faces of a given dimension --- of this simplicial complex, while the coefficients of $F_{m,n}(x,x)$ give the number of maximal faces, called facets, of a fixed dimension. Our work in this paper was originally motivated by wanting to determine how many facets there are with a given number of vertices representing Left or Right moves. We solve this problem by finding $F_{m,n}(x,y)$ in \cref{sec:maximal}.

\section{Counting all \domineering positions}\label{sec:general}

Given $m$ and $n$, our goal is to count the number of \domineering positions that can occur as the result of play in an $m \times n$ rectangle. To distinguish this problem from counting specific types of \domineering positions later we refer to these as \emph{general} \domineering positions. 

We will find a generating function $D_{m,n}(x,y)$ such that
\[
D_{m,n} (x,y) = \sum d(a,b) x^a y^b
\]
where $d(a,b)$ is the number of \domineering positions with $a$ Left (vertical) dominoes and $b$ Right (horizontal) dominoes.

We will extract from the generating function information on the play positions, those that can be reached through alternating play, in \cref{sec:playPositions}.

In their paper, Oh \cite{Oh2019} considers an equivalent problem to the particular problem of counting general positions using a trivariate generating function, where the third variable counts the number of empty squares. The number of empty squares is not particularly relevant to \Domineering. However, this information can be deduced from the number of dominoes of either player and the board size. We use an equivalent method, but as our generating functions are slightly different, our recursions are different as well. We include the recursion here for completeness, particularly as we continue using the matrices in \cref{sec:nonrectangular}.

To count rectangular \domineering positions we consider the tilings of rectangles using tiles with edge labels, then count the tilings that correspond to positions. 

We use the tiles in \cref{fig:general-tiles}.

\begin{figure}[!ht]
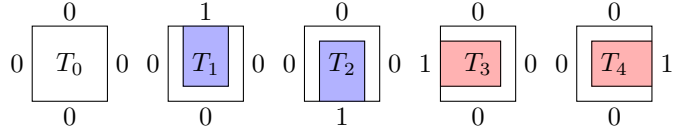

\begin{center}
\TO \TI \TII \TIII \TIV
\end{center}
\caption{The tiles used to count general \domineering positions}
\label{fig:general-tiles}
\end{figure}

For a rectangular tiling to form a \domineering position the tiling must satisfy the following conditions:
\begin{enumerate}
\item \textit{Adjacency condition}: All shared edges of adjacent tiles have the same label.
\item \textit{Boundary condition}: All boundary edges of the tiling have label 0.
\end{enumerate}

Together these conditions ensure that the tiling has no half-dominoes.

In turn, any \domineering position can be uniquely represented using such a tiling. See \cref{fig:general-position-tiling} for an example.

In the tiling of a \domineering position, such as in \cref{fig:general-position-tiling}, the number of Left (vertical) dominoes is equal to the number of $T_1$, and the number of Right (horizontal) dominoes is equal to the number of $T_3$.

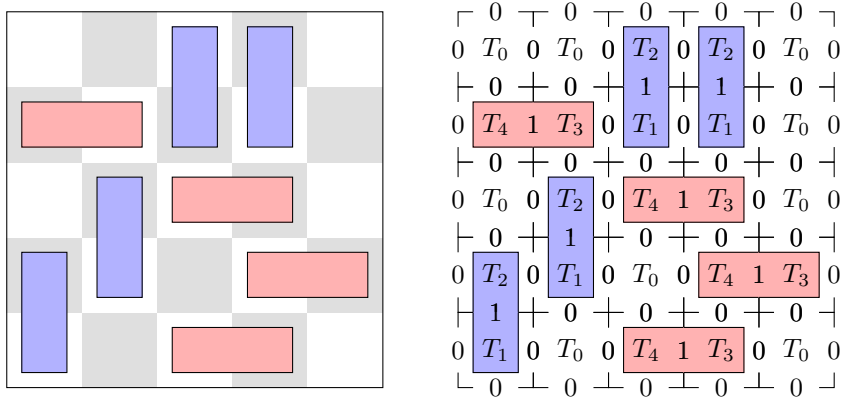
\begin{figure}[!ht]
\centering
\begin{tikzpicture}
\tikzchessboard[4]
\Ldomino{1}{1}
\Ldomino{2}{3}
\Rdomino{3}{1}
\Ldomino{0}{0}
\Rdomino{2}{0}
\Rdomino{0}{3}
\Ldomino{3}{3}
\Rdomino{2}{2}
\begin{scope}[shift={(6,0)}]
\TOboard{1}{0}
\TOboard{0}{2}
\TOboard{0}{4}
\TOboard{2}{1}
\TOboard{1}{4}
\TOboard{4}{0}
\TOboard{4}{2}
\TOboard{4}{3}
\TOboard{4}{4}
\TIIboard{0}{1}
\TIIboard{1}{2}
\TIIboard{2}{4}
\TIIboard{3}{4}
\TIboard{0}{0}
\TIboard{1}{1}
\TIboard{2}{3}
\TIboard{3}{3}
\TIVboard{2}{0}
\TIVboard{3}{1}
\TIVboard{2}{2}
\TIVboard{0}{3}
\TIIIboard{3}{0}
\TIIIboard{4}{1}
\TIIIboard{3}{2}
\TIIIboard{1}{3}
\end{scope}
\end{tikzpicture}
\caption{A \domineering position on a $5\times 5$ board and its equivalent tiling.}
\label{fig:general-position-tiling}
\end{figure}

We use the term \textbf{mosaic} for tilings that satisfy the adjacency condition, but not necessarily the boundary condition. A $1 \times q$ mosaic is called a \textbf{bar mosaic}, that is, a row of $q$ tiles.

We will find generating polynomials by first counting bar mosaics. We then stack these bars to count rectangular mosaics; in this step that we consider the adjacency condition for vertically adjacent tiles. Finally, we restrict to those tilings that also satisfy the boundary condition.

We will count mosaics instead of just \domineering tilings as a row or column of a \domineering rectangle is not necessarily a \domineering position itself. For example in \cref{fig:general-position-tiling} the second row from the bottom forms the bar mosaic shown in \cref{fig:general-bar-mosaic} --- it does not correspond to a \domineering position on its own as the leftmost two tiles are unpaired half-dominoes, but shared edges of horizontally adjacent tiles do have the same label.

\begin{figure}[!ht]
\centering\begin{tikzpicture}
\TOboard{2}{0}
\TIboard{1}{0}
\TIIboard{0}{0}
\TIIIboard{4}{0}
\TIVboard{3}{0}
\end{tikzpicture}
\caption{A row of a \domineering tiling is a bar mosaic, which satisfies the (horizontal) adjacency condition but not necessarily boundary condition.}
\label{fig:general-bar-mosaic}
\end{figure}
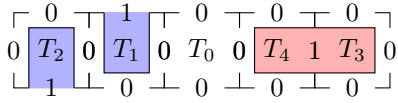

The enumeration of bar mosaics will be done using matrices, called \textbf{bar-state matrices}. We take care to construct these matrices so that we can count rectangular mosaics by matrix multiplication.  As we only want to count mosaics that satisfy the boundary condition in the end, we can make due with counting only bar mosaics with label 0 on the right end. The enumeration of the bar mosaics will be done recursively --- we use the enumeration of bar mosaics to enumerate longer bar mosaics --- so our matrices are recursively defined (using blocks). We use two bar-state matrices; the matrix to count bar mosaics of length $q$ with right label 0 and left label 0 is denoted by $G_{0,q}$ and the matrix to count bar mosaics of length $q$ with right label 0 and left label 1 is denoted by $G_{1,q}$. We use $G$ here for naming the matrices, standing for `general', to reserve $M$ for matrices in the subsequent sections of the paper, which include a discussion of maximal positions.

\begin{theorem}\label{thm:generalPositions}
The polynomial profile of \domineering on an $m\times n$ board is the (1,1) entry of $G_{0,n}^m$ where $G_{0,0} = \begin{bmatrix} 1\\ \end{bmatrix}$,
$G_{1,0} = \begin{bmatrix} 0\\ \end{bmatrix}$,

\[G_{0,q+1} =
\left[\begin{tabular}{R{1cm}R{1cm}}
$G_{0,q}$ $+G_{1,q}$ & $xG_{0,q}$ \\
$G_{0,q}$ & \zeroblock
\end{tabular}\right], \text{ and }
G_{1,q+1} =
\left[\begin{tabular}{R{1cm}R{1cm}}
$yG_{0,q}$ & \zeroblock \\
\zeroblock & \zeroblock
\end{tabular}\right].
\]

\end{theorem}

\begin{proof}
We will first prove that the bar-state matrices give the generating polynomials for bar mosaics.

In a bar-state matrix the rows and columns are indexed by binary strings of length $q$ and thus each matrix has size $2^q\times 2^q$. We order these strings lexicographically. Thus the strings of length 3 are ordered as
\[000,001,010,011,100,101,110,111.\]

An entry in a bar-state matrix is the generating polynomial of the bar whose top edge labels are the column index, and whose bottom edge labels are the row index; the right edge label is 0 and the left edge label corresponds to the matrix (the first subscript of $G$). The generating polynomial of the bar is the monomial $x^ay^b$ where $a$ is the number of $T_1$ in the bar and $b$ is the number of $T_3$ in the bar.

For $q=0$, the length is 0, it is only possible to satisfy the adjacency condition if both the starting and ending labels are the same. We thus have $G_{0,0} = \begin{bmatrix} 1 \end{bmatrix}$ and $G_{1,0} = \begin{bmatrix} 0 \end{bmatrix}$.

Our construction for larger $q$ uses blocks and we name them as in \cref{fig:4blocklabels} to avoid proliferating subscripts. We need to show that in each matrix the four blocks are as we claim. Our argument is recursive and follows by first considering the leftmost tile in each bar. The block considered determines the top label and bottom label of the leftmost tile in bar mosaics enumerated as in \cref{fig:4blocklabels}.

\begin{figure}[ht]
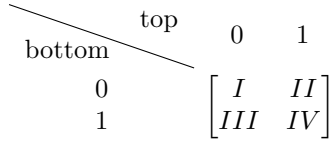

\[\begin{array}{c c} 
	\mbox{\diagbox[width=7em]{bottom}{top}} & \begin{array}{c c} 0\ &  \ \ \ 1 \\ \end{array} \\[2ex]
	\begin{array}{c c}
	0\\
	1
	\end{array}
	&
	\begin{bmatrix}
	I & II\\
	III & IV
	\end{bmatrix}
\end{array}\]
\caption{The block construction of the state-matrices for general \domineering positions.}
\label{fig:4blocklabels}
\end{figure}

For $q=1$, the bars we are counting consist of a single tile which must have right label 0 (or cannot have right label 1), that is, one of the four tiles $0$--$3$; you may wish to reference \cref{fig:general-tiles}. Thus between the bar-state matrices $G_{0,1}$ and $G_{1,1}$ we have four nonzero blocks --- determined by the top and bottom labels of tiles $0$--$3$. Recall that each block has a corresponding top and bottom label of the leftmost tile.
Thus 
\[G_{0,1}=\begin{bmatrix}
1 & x\\ 1 & 0
\end{bmatrix}
\mbox{ and }
G_{1,1}=\begin{bmatrix}
y & 0\\ 0 & 0
\end{bmatrix},\]
which shows that the recursion is correct for the first step.

The bar-state matrix $G_{0,q+1}$ enumerates mosaics where the leftmost tile has left label 0 (and right label 0). We consider possible leftmost tiles which are tiles $0$--$2$ and $4$, as these are the only tiles with left label 0.
Of these four tiles, those tiles that meet the conditions for block I are tile $0$ and tile $4$. As the right label of tile $0$ is 0, the possible completions of our bar when the leftmost tile is tile $0$ are given by $G_{0,q}$. As the right label of tile $4$ is 1, the possible completions of our bar are given by $G_{1,q}$. So block I is $G_{0,q}+G_{1,q}$.

The conditions for block II (that the leftmost tile has top label 1 and bottom label 0) require the leftmost tile in the bar mosaic to be $T_1$ and the entry is thus $xG_{0,q}$ because the right label of tile $1$ is 0 and $G_{0,q}$ counts the bar mosaics with left label 0 and right label 0 (recall that for each $T_1$ we have an $x$).

The conditions for block III force $T_2$ as the leftmost tile, giving the entry $G_{0,q}$. For block IV, we would need a tile with bottom and top label $1$ but no tiles have such labels and so our entry is $\zero$.

In $G_{1,q+1}$, we have a left label of 1. As our recursion starts by consider possible left tiles, we need only consider tile $3$, as this is the only tile with 1 as a left label.

Thus the only block for which tile $3$ has appropriate top and bottom labels is block I. Since $T_3$ is forced and its left label is 0, the possible completions of the bar are given by $G_{0,q}$ and the entry in block one is $yG_{0,q}$. All other blocks, having no possible leftmost tiles, count no mosaics and thus their blocks are $\zero$.

We turn our attention to rectangular $m \times n$ tilings and in particular those that satisfy the left and right boundary conditions, so we will only consider the bar mosaics of length $n$ with left label 0 in addition to having the right label 0; these are counted by  $G_{0,n}$.

First consider $2\times n$ mosaics that satisfy the adjacency condition.
Such a mosaic consists of two bar mosaics where the string of bottom labels of one mosaic matches the string of top labels of the other. Suppose $A$ is the bar-state matrix for the top bar and $B$ is the bar-state matrix for the bottom bar. To count all the $2\times n$ mosaics that satisfy the adjacency condition we need to consider all $2^n$ strings of matched labels along the stacked bar mosaics. That is, we take columns in $A$ and rows in $B$. Since the other labels of the bar mosaics have to match, we take the dot product of the two vectors to get the polynomial for this mosaic. Thus the matrix for the $2\times n$ mosaic is $BA$.

To get mosaics that also satisfy the boundary conditions and thus correspond to \domineering positions we can start with bar mosaics that satisfy the left and right boundary conditions; in our case $A=B=G_{0,n}$, so the state matrix for $2 \times n$ mosaics where the left and right boundary conditions are also satisfied is $G_{0,n}^2$.

For larger mosaics this works the same way; we stack more bars and find that the matrix for a $m\times n$ mosaic is $G_{0,n}^m$.

Fixing the string of top labels and of bottom labels for the $2\times n$ mosaic corresponds to specifying an entry in the matrix.
Finally, to satisfy the boundary condition on top and bottom we need to restrict to all top and bottom labels being 0. This means we take the entry in the top left corner, entry $(1,1)$.
\end{proof}

\begin{example}\label{ex:4x3}
Now we compute the number of general positions on a $4\times 3$ \domineering board. We need the bar-state matrix for bar mosaics of length $3$ satisfying the left and right boundary conditions, $G_{0,3}$, which we find recursively to be as follows:

\[G_{0,1}=\begin{bmatrix}
1 & x\\ 1 & 0
\end{bmatrix}\qquad
G_{1,1}=\begin{bmatrix}
y & 0\\ 0 & 0
\end{bmatrix}\]
\[G_{0,2}=\left[\begin{array}{cc:cc}
y+1 & x & x & x^2\\
1 & 0 & x & 0\\
\hdashline
1 & x & 0 & 0\\
1 & 0 & 0 & 0
\end{array}\right]\qquad
G_{1,2}=\left[\begin{array}{cc:cc}
y & xy & 0 & 0\\
y & 0 & 0 & 0\\
\hdashline
0 & 0 & 0 & 0\\
0 & 0 & 0 & 0
\end{array}\right]\]
\[G_{0,3}=\left[\begin{array}{cccc:cccc}
2y+1 & xy+x & x & x^2 & xy+x & x^2 & x^2 & x^3\\
y+1 & 0 & x & 0 & x & 0 & x^2 & 0\\
1 & x & 0 & 0 & x & x^2 & 0 & 0\\
1 & 0 & 0 & 0 & x & 0 & 0 & 0\\
\hdashline
y+1 & x & x & x^2 & 0 & 0 & 0 & 0\\
1 & 0 & x & 0 & 0 & 0 & 0 & 0\\
1 & x & 0 & 0 & 0 & 0 & 0 & 0\\
1 & 0 & 0 & 0 & 0 & 0 & 0 & 0
\end{array}\right]
\]

Calculating $G_{0,3}^4$ gives the state matrix for $4 \times 3$ boards and taking the $(1,1)$ entry, gives the generating function for $4 \times 3$ \domineering positions as follows:
\begin{align*}
D_{4,3}(x,y) &= x^6 + 9x^5 + 6x^4y^2 + 20x^4y + 30x^4 + 46x^3y^2 + 84x^3y + 45x^3\\
&\qquad + 4x^2y^4 + 24x^2y^3 + 100x^2y^2 + 100x^2y + 30x^2+24xy^4\\
&\qquad + 72xy^3 + 90xy^2 + 48xy + 9x + 16y^4 + 32y^3 + 24y^2+ 8y + 1.
\end{align*}

For example, on a $4\times 3$ \domineering board there are $90$ positions with $1$ vertical domino and $2$ horizontal dominoes as shown by the term $90xy^2$ and there is $1$ position with $6$ vertical dominoes as shown by the term $x^6$.

\end{example}

\subsection{Play positions}\label{sec:playPositions}

The polynomials we have found count all legal \domineering positions. As mentioned previously, we do not assume alternating play. Positions which may occur during alternating play are called \textbf{play positions}. To enumerate play positions we restrict $D_{m,n}(x,y)$ to include only terms where the difference in the powers of $x$ and $y$ is at most $1$. 

\begin{example}
Using \cref{ex:4x3} we can see that the polynomial for the play positions on a $4\times 3$ board is
\[46x^3y^2+ 24x^2y^3+ 100x^2y^2 + 100x^2y+ 90xy^2 + 48xy + 9x + 8y + 1.\]

Setting $x=y=1$ in this polynomial we get the total number of play positions, which in this case is $426$.
\end{example}

The number of play positions on an $n \times n$ board is in the OEIS as \seqnum{A332714}. The terms up to $n=10$ are given in \cref{tab:number-of-play-positions}. We also include the ratio of all positions that are play positions, truncated to $5$ decimals. Recall that the number of all positions is the number of matchings in a grid graph, which is given in the OEIS as \seqnum{A028420}.

\begin{table}[ht]

\centering

$\begin{array}{r||r||c}
n & \mbox{Number of play positions} & \mbox{Ratio to all positions}\\\hline\hline
1 & 1 & 1 \\\hline
2 & 5 & 0.71428 \\\hline
3 & 75 & 0.57251 \\\hline
4 & $4,632$ & 0.46264 \\\hline
5 & $1,076,492$ & 0.38299 \\\hline
6 & $963,182,263$	& 0.32222 \\\hline
7 & $3,317,770,165,381$ & 0.27774 \\\hline
8 & $43,809,083,383,524,391$ & 0.24367 \\\hline
9 & $2,209,112,327,971,366,587,064$ & 0.21689 \\\hline
10 & $424,273,291,301,040,427,702,718,109$ & 0.19532\\
\end{array}$
\caption{The number of play positions and ratio of play positions to all positions on small square boards.}
\label{tab:number-of-play-positions}
\end{table}

The ratios computed in \cref{tab:number-of-play-positions} give a sense of the cost of considering all plays from a position, as is standard in combinatorial game theory, compared to considering only plays reachable in alternating play, which is common in practice.

We expect that the ratio of play positions decreases as the board size increases because on larger boards there are an increasing number of ways to have a board with the number of dominoes for the two players differing by more than one. The ratios apparently do not follow an arithmetic or geometric sequence. 

We expect the ratios of play positions to general positions to approach $0$. While it is interesting to consider how this ratio changes as the board size grows, it is also potentially interesting to observe how this ratio changes during play.

\section{Non-rectangular boards}\label{sec:nonrectangular}

The techniques we use to find the generating polynomials of rectangular boards can be used to find the generating polynomials of non-rectangular boards as well.

We think of a non-rectangular board as being a rectangular board with missing squares and treat the missing squares as squares forced to be empty. For example, the board on the left in \cref{fig:non-rectangular-board} we think of as being contained in a $2\times 3$ board with forced empty squares in the top left and top middle.

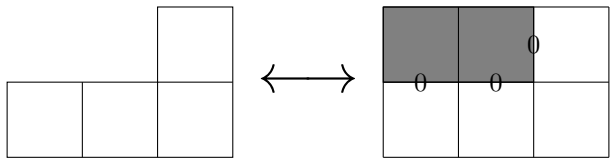
\begin{figure}[!ht]
\centering
\begin{tikzpicture}
\draw (0,0)--(3,0)--(3,2)--(2,2)--(2,1)--(0,1)--cycle;
\draw (1,0)--(1,1);
\draw (2,0)--(2,1)--(3,1);

\node at (4,1) {\huge $\longleftrightarrow$};
\begin{scope}[shift={(5,0)}]
\filldraw[fill=gray] (0,1)--(2,1)--(2,2)--(0,2)--cycle;
\foreach \x in {0,1,2,3}{
\draw (\x,0)--(\x,2);
}
\foreach \y in {0,1,2}{
\draw (0,\y)--(3,\y);
}
\node at (0.5,1) {$0$};
\node at (1.5,1) {$0$};
\node at (2,1.5) {$0$};
\end{scope}
\end{tikzpicture}
\caption{The equivalence of a non-rectangular board and a rectangular board with forced empty squares.}
\label{fig:non-rectangular-board}
\end{figure}

We will create a (bar-state) matrix $R_i$ for each row, then multiply these matrices to get the state matrix for the entire mosaic. For the board in \cref{fig:non-rectangular-board} the matrix for row $1$ is $R_1=G_{0,1}$. For row 2 we will use $G_{0,3}$, but restrict those columns which have 0 in the first and second position of the top label. We denote this as $R_2=\left.G_{0,3}\right|^{00\_}$. For restrictions on bottom labels we will use a subscript. \cref{ex:3x4-nonrectangular-board} shows this procedure, but before that we consider rows with missing squares in the middle as in \cref{ex:row-with-missing-square}. If a row contains an missing square in the middle, the Kronecker product of the bar-state matrices of the shorter strips is the bar-state matrix for the row. The \textbf{Kronecker product} of two matrices $A$ and $B$ is
\[A\otimes B=\begin{bmatrix}
a_{11}B & \cdots & a_{1n}B\\
\vdots & \ddots & \vdots\\
a_{m1}B & \cdots & a_{mn}B
\end{bmatrix}.\]

If missing squares in a row break it into three or more parts we use the Kronecker product multiple times.

The tiling (and the play) in the left part of the broken row is independent of the tiling (and play) in the right part. The state-matrix for the whole row is a block matrix where the blocks are determined by possible labels of the first part of the row, while the order of the second part in each block is the same and determined by the possible top and bottom labels in that part. Thus, taking the Kronecker product of the matrix for the first strip with the matrix for the second strip is exactly what we need here.

\begin{example}\label{ex:row-with-missing-square}
Consider the strip 
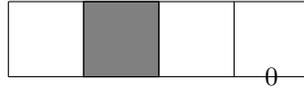
\begin{center}
\begin{tikzpicture}
\filldraw[fill=gray] (1,0)--(1,1)--(2,1)--(2,0)--cycle;
\foreach \x in {0,1,2,3,4}{
\draw (\x,0)--(\x,1);
}
\foreach \y in {0,1}{
\draw (0,\y)--(4,\y);
}
\node at (3.5,0) {$0$};
\end{tikzpicture}
\end{center}
which is the center row of the board in \cref{fig:exampleBoard}.

The matrix for this row is 
\begin{align*}
G_{0,1}\otimes G_{0,2}|_{\_0} &=
\begin{bmatrix}
1 & x\\
1 & 0
\end{bmatrix}\otimes
\begin{bmatrix}
y+1 & x & x & x^2\\
1 & x & 0 & 0
\end{bmatrix}\\
&=
\left[\begin{array}{cccc:cccc}
y+1 & x & x & x^2 & xy+x & x^2 & x^2 & x^3\\
1 & x & 0 & 0 & x & x^2 & 0 & 0\\
\hdashline
y+1 & x & x & x^2 & 0 & 0 & 0 & 0\\
1 & x & 0 & 0 & 0 & 0 & 0 & 0
\end{array}\right]
.
\end{align*}

\end{example}

Finally, we multiply the matrices for each row such that the bottom row matrix is the left-most one, while the top row matrix is the right-most one, allowing us to match the top/column label of each row with the bottom/row label of the row above.

\begin{example}\label{ex:3x4-nonrectangular-board}
Consider the board in \cref{fig:exampleBoard}.

\begin{figure}[!ht]
\begin{center}
\begin{tikzpicture}
\filldraw[fill=gray] (1,1)--(1,2)--(2,2)--(2,1)--cycle;
\filldraw[fill=gray] (3,0)--(3,1)--(4,1)--(4,0)--cycle;
\foreach \x in {0,1,2,3,4}{
\draw (\x,0)--(\x,3);
}
\foreach \y in {0,1,2,3}{
\draw (0,\y)--(4,\y);
}
\end{tikzpicture}
\end{center}
\caption{An example non-rectangular board}
\label{fig:exampleBoard}
\end{figure}
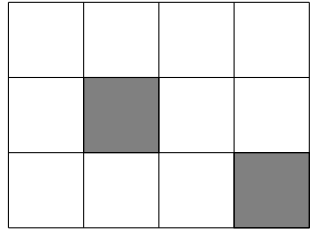

For row 1 we have
\[R_1=G_{0,4}|_{\_0\_\_};
\]
for row 2, as in \cref{ex:row-with-missing-square},
\[R_2=G_{0,1}\otimes G_{0,2}|_{\_0}; \]
and for row 3
\[R_3=G_{0,3}|^{\_0\_}.\]

The matrix for the entire mosaic is then $G=R_3\times R_2\times R_1$. To apply the restriction of having only $0$s for top and bottom labels we take only the $(1,1)$ entry of the matrix. Thus we get the generating polynomial
\begin{align*}
D&=2x^3y^2 + 8x^3y + 4x^3 + 12x^2y^2 + 22x^2y + 8x^2 + xy^4 + 10xy^3 + 26xy^2\\
&\qquad + 21xy + 5x + 2y^4 + 9y^3 + 12y^2 + 6y + 1.
\end{align*}
\end{example}

Using this technique for non-rectangular boards, we are thus able to find the generating polynomial counting positions that can be reached from a partially played board. We will do so in \cref{ex:GONC}, where we will also use the following simplification for disjunctive sums.

\begin{proposition}
Suppose $G$ is the game of \domineering on the board $B_1$ and $H$ is \domineering on $B_2$. Then the generating polynomial of $G+H$ is the product of the generating polynomials of $G$ and $H$.
\end{proposition}
\begin{proof}
Let the coefficient of $x^iy^j$ in the generating polynomial of $G$ be $a_{i,j}$, of $H$ be $b_{i,j}$, and of $G+H$ be $c_{i,j}$.

Consider a position in $G+H$ with $u$ Left pieces and $v$ Right pieces. Say $i$ Left pieces and $j$ Right pieces of these have been played in $G$ and the rest in $H$. There are $a_{i,j}b_{u-i,v-j}$ such positions, and thus \[c_{u,v}=\sum_{\substack{0\leq i\leq u\\ 0\leq j\leq v}} a_{i,j}b_{u-i,v-j}.\] Therefore, the generating polynomial of $G+H$ is the product of those of $G$ and $H$.
\end{proof}

This proposition again emphasizes that it makes sense to enumerate all positions, not just play positions, due to disjunctive sums allowing a difference in the number of pieces larger than 1. 

\begin{example}\label{ex:GONC}
The position in \cref{fig:GONCgame} occurred partway through Game 1 of the finals of the \domineering tournament at the 1994 workshop ``Games of No Chance'' at MSRI (see the tournament report by West \cite[Fig.\ 3]{West1996}). We will determine the polynomial for all play positions that would have to be analyzed from this point on. 

\begin{figure}[!ht]
\begin{center}
\begin{tikzpicture}
\tikzchessboard[7]
\LdominoDark{1}{0}
\LdominoDark{1}{2}
\LdominoDark{1}{4}
\LdominoDark{3}{2}
\LdominoDark{6}{2}
\LdominoDark{6}{6}
\RdominoDark{0}{6}
\RdominoDark{2}{1}
\RdominoDark{4}{4}
\RdominoDark{4}{6}
\RdominoDark{6}{1}
\draw[color=red] (0.1,7.1)--(0.1,7.9)--(5.9,7.9)--(5.9,7.1)--(3.9,7.1)--(3.9,5.9)--(7.1,5.9)--(7.1,7.9)--(7.9,7.9)--(7.9,2.1)--(7.1,2.1)--(7.1,4.1)--(6.1,4.1)--(6.1,5.1)--(3.9,5.1)--(3.9,4.1)--(2.9,4.1)--(2.9,2.1)--(2.1,2.1)--(2.1,7.1)--(0.1,7.1);
\draw[color=blue] (2.1,0.1)--(2.1,0.9)--(4.1,0.9)--(4.1,3.9)--(5.9,3.9)--(5.9,0.9)--(7.9,0.9)--(7.9,0.1)--(2.1,0.1);
\draw[color=green!70!black] (0.1,0.1)--(0.1,5.9)--(0.9,5.9)--(0.9,0.1)--(0.1,0.1);
\end{tikzpicture}
\end{center}
\caption{A board position from Game 1 of the 1994 tournament finals with disjunctive components highlighted.}
\label{fig:GONCgame}
\end{figure}
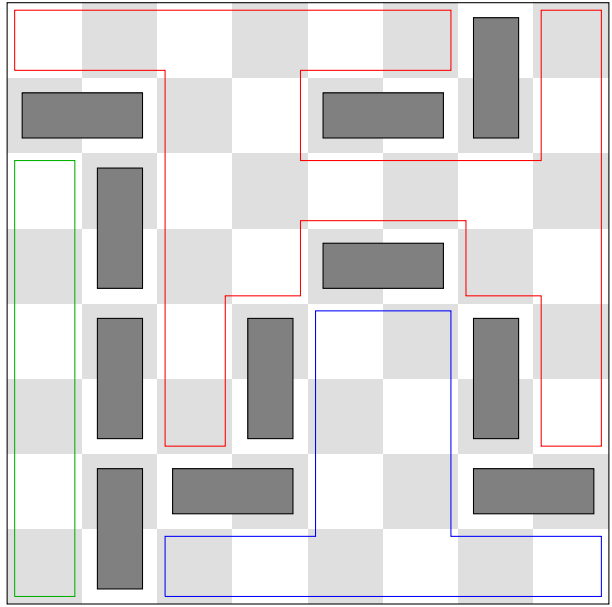

We start by finding the generating polynomials for general positions of each of the three disjunctive sum component as we have done in \cref{ex:3x4-nonrectangular-board}.

The ``bottom'' component, bordered in blue, as a mosaic has the matrix
\[G_{0,6}|^{00\_\,\_00}\times G_{0,2}^3\]
(recall that we multiply the matrices for each row from the bottom up). Thus the generating polynomial, the $(1,1)$ entry of this matrix, is \begin{align*}
&x^4y^2 + 2x^4y + x^4 + 10x^3y^2 + 16x^3y + 6x^3 + 3x^2y^4 + 24x^2y^3 + 58x^2y^2\\
&\qquad + 46x^2y + 11x^2 + 8xy^4 + 42xy^3 + 62xy^2 + 34xy + 6x + y^6 + 9y^5\\
&\qquad + 26y^4 + 35y^3 + 24y^2 + 8y +1.
\end{align*}

The ``top'' component, bordered in red, has the matrix 
\begin{align*}
&(G_{0,1}\otimes G_{0,1})\times( G_{0,1}\otimes G_{0,1})\times( G_{0,2}|_{\_0}\otimes G_{0,2}|_{0\_})\times(G_{0,6}|_{\_\,\_00\_\,\_}^{\_\,\_000\_})\\
&\quad\times(G_{0,2}\otimes G_{0,1})\times(G_{0,6}|_{00\_\,\_00}\otimes G_{0,1}).
\end{align*}

The generating polynomial is 
\begin{align*}
&x^9y^3 + 3x^9y^2 + 3x^9y + x^9 + 23x^8y^3 + 62x^8y^2 + 55x^8y + 16x^8 + 4x^7y^5\\
&\quad + 43x^7y^4 + 291x^7y^3 + 555x^7y^2 + 401x^7y + 98x^7 + 60x^6y^5 + 539x^6y^4\\
&\quad + 1874x^6y^3 + 2564x^6y^2 + 1469x^6y + 295x^6 + 4x^5y^7 + 71x^5y^6 + 650x^5y^5\\
&\quad + 2908x^5y^4 + 6250x^5y^3 + 6303x^5y^2 + 2885x^5y +481x^5 + 34x^4y^7\\
&\quad + 467x^4y^6 + 2590x^4y^5 + 7406x^4y^4 + 11148x^4y^3 + 8637x^4y^2 + 3218x^4y\\
&\quad + 452x^4 + x^3y^9 + 20x^3y^8 + 222x^3y^7 + 1421x^3y^6 + 4991x^3y^5\\
&\quad + 9909x^3y^4 + 11109x^3y^3 + 6812x^3y^2 + 2098x^3y + 251x^3 + 3x^2y^9 + 55x^2y^8\\
&\quad + 437x^2y^7 + 1909x^2y^6 + 4790x^2y^5 + 7083x^2y^4 + 6172x^2y^3 + 3063x^2y^2\\
&\quad + 789x^2y + 81x^2 + 3xy^9 + 50xy^8 + 330xy^7 + 1123xy^6 + 2181xy^5\\
&\quad + 2533xy^4 + 1774xy^3 + 726xy^2 + 158xy + 14x + y^9 + 15y^8 + 85y^7 + 237y^6\\
&\quad + 373y^5 + 353y^4 + 204y^3 + 70y^2 + 13y + 1.
\end{align*}

Finally, the component on the left, bordered in green, is simply a $6\times 1$ board, thus having matrix $G_{0,1}^6$ and generating function $D_{1,6}(x,y)=x^3 + 6x^2 + 5x + 1$.

Multiplying all three generating polynomials we get the generating polynomial for the entire position. Restricting to only play positions this is
\begin{align*}
&17x^{10}y^{11} + 410x^{10}y^{10} + 7690x^{10}y^9 + 6769x^9y^{10} + 76829x^9y^9 + 532379x^9y^8 \\
&\quad +  436560x^8y^9 + 2217847x^8y^8 + 7453953x^8y^7 + 6030494x^7y^8\\
&\quad + 16049771x^7y^7 + 29420257x^7y^6 + 23698832x^6y^7 + 36239078x^6y^6\\
&\quad + 38789964x^6y^5 + 31354701x^5y^6 + 29013063x^5y^5 +18784523x^5y^4\\
&\quad + 15287335x^4y^5 + 8768628x^4y^4 + 3451191x^4y^3 +2830886x^3y^4\\
&\quad + 1004132x^3y^3 + 232953x^3y^2  + 192647x^2y^3 + 40779x^2y^2 + 5086x^2y\\
&\quad + 4240xy^2 + 487xy + 25x + 21y + 1.
\end{align*}

In this polynomial we see from the small degree terms the number of possible paths of the game from this position. The high degree terms give some sense of the state at the end of the game provided play continues in such a way that the game lasts as long as possible. In this game Left played first; we see that Left (vertical) has played 6 times, whereas Right has played 5 times, and thus it is Right's turn from this position. Right has $21$ possible moves. Right would likely have been happy to see terms in the polynomial where the degree of $y$ is greater than the degree of $x$ as these indicate the possibility of Right winning; of note then is the term $17x^{10}y^{11}$ with a relatively small coefficient. The actual winner of this game was Left.
\end{example}

Further, by setting $x=y$ in all matrices, we are also able to count the number of matchings of a subgraph of a grid using this technique.

\section{Counting maximal \domineering positions}\label{sec:maximal}
	
Recall that a maximal \domineering positions is one in which neither player has any available moves. In this section we count maximal \domineering positions. We find the generating function
\[F_{m,n}(x,y)=\sum f(a,b)x^ay^b\]
where the board has size $m\times n$ and $f(a,b)$ is the number of maximal \domineering positions with $a$ Left (vertical) dominoes and $b$ Right (horizontal) dominoes. We use $F$ for the generating polynomial of the maximal positions as these correspond to the \textit{facets} (maximal faces) of the simplicial complex representing the position.

We will use the same technique as for the general \domineering positions of counting bar mosaics, multiplying the matrices, and then restricting to tilings, with only a few small differences. The tiles will be different so that with only using adjacency and boundary conditions we can force the tiling to be equivalent to a maximal position, i.e., no two empty spaces may be adjacent, and in turn a position corresponds to a unique tiling. The possible tiles in this case are given in \cref{fig:maximal-tiles}.
\begin{figure}[!ht]
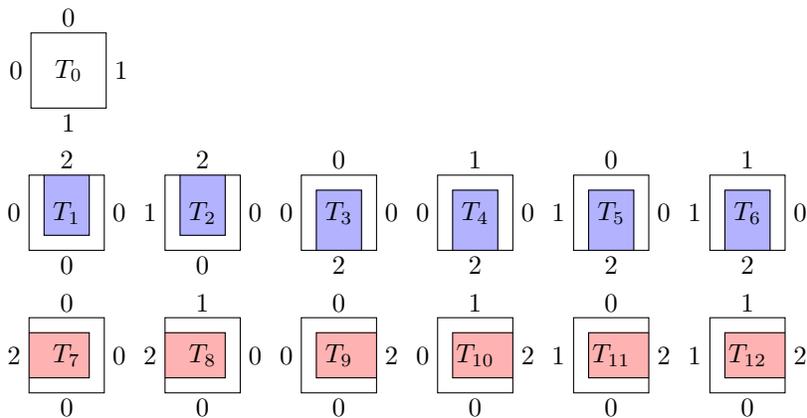

\begin{center}
\MaxTO\hspace{9cm}

\MaxTI \MaxTII \MaxTIII \MaxTIV \MaxTV \MaxTVI

\MaxTVII \MaxTVIII \MaxTIX \MaxTX \MaxTXI \MaxTXII
\end{center}
\caption{The tiles used to count maximal \domineering positions.}
\label{fig:maximal-tiles}
\end{figure}

For a rectangular tiling to form a maximal \domineering position the tiling must satisfy the following conditions:
\begin{enumerate}
\item \textit{Adjacency condition}: All shared edges of adjacent tiles have the same label.
\item \textit{Boundary condition}: The left and top boundary edges of the tiling have label 0; the right and bottom boundary edges have label 0 or 1.
\end{enumerate}

In turn, any maximal \domineering position can be uniquely represented using a such a tiling. To decide between the different tiles representing domino halves one simply has to check for empty squares above and to the left. As an example, see the position and equivalent tiling in \cref{fig:maximal-position-tiling}.

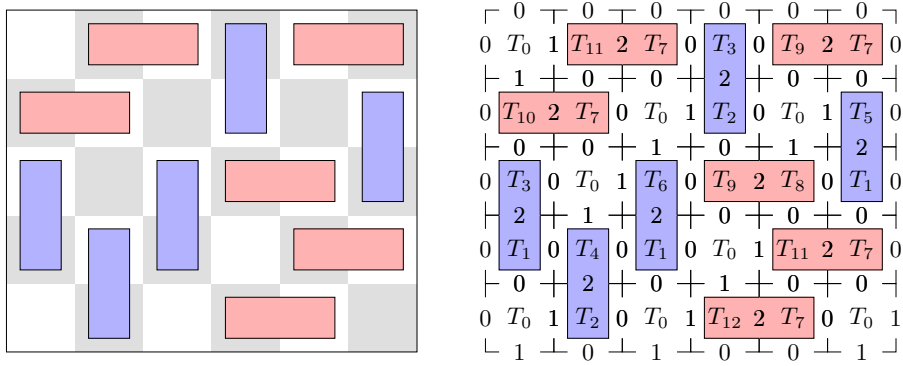
\begin{figure}[!ht]
\centering
\resizebox{\linewidth}{!}{
\begin{tikzpicture}
\foreach \x in {0,...,5} \foreach \y in {0,...,4}
	{
		\pgfmathparse{mod(\x+\y+6,2) ? "gray!25" : "white"}
		\edef\colour{\pgfmathresult}
		\path[fill=\colour] (\x,\y) rectangle ++ (1,1);
	}
	\draw (0,0)--(0,5)--(6,5)--(6,0)--cycle;
	
\Ldomino{0}{1}
\Ldomino{1}{0}
\Ldomino{2}{1}
\Ldomino{3}{3}
\Ldomino{5}{2}
\Rdomino{3}{0}
\Rdomino{4}{1}
\Rdomino{3}{2}
\Rdomino{0}{3}
\Rdomino{1}{4}
\Rdomino{4}{4}
\begin{scope}[shift={(7,0)}]
\MaxTOboard{0}{0}
\MaxTOboard{0}{4}
\MaxTOboard{1}{2}
\MaxTOboard{2}{0}
\MaxTOboard{2}{3}
\MaxTOboard{3}{1}
\MaxTOboard{4}{3}
\MaxTOboard{5}{0}

\MaxTIboard{0}{1}
\MaxTIboard{2}{1}
\MaxTIboard{5}{2}
\MaxTIIboard{1}{0}
\MaxTIIboard{3}{3}

\MaxTVIIboard{1}{3}
\MaxTVIIboard{2}{4}
\MaxTVIIboard{4}{0}
\MaxTVIIboard{5}{1}
\MaxTVIIboard{5}{4}
\MaxTVIIIboard{4}{2}

\MaxTIIIboard{0}{2}
\MaxTIIIboard{3}{4}
\MaxTIVboard{1}{1}
\MaxTVboard{5}{3}
\MaxTVIboard{2}{2}

\MaxTIXboard{3}{2}
\MaxTIXboard{4}{4}
\MaxTXboard{0}{3}
\MaxTXIboard{1}{4}
\MaxTXIboard{4}{1}
\MaxTXIIboard{3}{0}

\end{scope}
\end{tikzpicture}}
\caption{A maximal \domineering position on a $5\times 6$ board and its equivalent tiling.}
\label{fig:maximal-position-tiling}
\end{figure}

Note that the number of Left (vertical) dominoes is equal to the number of $T_1$ and $T_2$ combined, and the number of Right (horizontal) dominoes is equal to the number of $T_7$ and $T_8$ combined.

In this case we will use $M$ in the name of the bar-state matrices to represent that we are counting \textit{maximal} positions only. Also note that, since the bottom and right edge in a tiling are allowed to have both 0 and 1 as their labels, the generating polynomial will be a sum of entries. A detailed explanation is given in the proof of the following theorem. 

\begin{theorem}\label{thm:MaxPositions}
The generating function for the maximal position of an $m \times n$ \domineering board is 
\[
F_{m,n}(x,y)=\mathlarger{\mathlarger{\mathlarger{\sum}}}_{u \in\{0,1\}^n} \Bigg(M_{0,n} + M_{0,n}'\Bigg)^m(1 + \sum_{i=1}^n u_i 3^i , 1)
\]
where $M_{0,0} = \begin{bmatrix} 1\\ \end{bmatrix}$,
$M_{1,0} = \begin{bmatrix} 0\\ \end{bmatrix}$,
$M_{2,0} = \begin{bmatrix} 0\\ \end{bmatrix}$, 
$M_{0,0}' = \begin{bmatrix} 0 \end{bmatrix}$,
$M_{1,0}' = \begin{bmatrix} 1 \end{bmatrix}$,
$M_{2,0}' = \begin{bmatrix} 0 \end{bmatrix}$,

\begin{center}
\resizebox{\textwidth}{!}{
\begin{minipage}{1.2\textwidth}
$M_{0,(q+1)} =
\left[\begin{tabular}{R{1cm}R{1cm}R{1.1cm}}
$M_{2,q}$ & $M_{2,q}$ & $xM_{0,q}$ \\
$M_{1,q}$ & \zeroblock & \zeroblock \\
$M_{0,q}$ & $M_{0,q}$ & \zeroblock
\end{tabular}\right]
$,\hspace{0.5cm}
$M_{0,(q+1)}' =
\left[\begin{tabular}{R{1cm}R{1cm}R{1.1cm}}
$M_{2,q}'$ & $M_{2,q}'$ & $xM_{0,q}'$ \\
$M_{1,q}'$ & \zeroblock & \zeroblock \\
$M_{0,q}'$ & $M_{0,q}'$ & \zeroblock
\end{tabular}\right],
$
\end{minipage}}

\bigskip

\resizebox{\textwidth}{!}{
\begin{minipage}{1.2\textwidth}
$M_{1,(q+1)} = 
\left[\begin{tabular}{R{1cm}R{1cm}R{1cm}}
$M_{2,q}$ & $M_{2,q}$ & $xM_{0,q}$ \\
\zeroblock & \zeroblock & \zeroblock \\
$M_{0,q}$ & $M_{0,q}$ & \zeroblock
\end{tabular}\right]
$,\hspace{0.5cm}
$M_{1,(q+1)}' = 
\left[\begin{tabular}{R{1cm}R{1cm}R{1cm}}
$M_{2,q}'$ & $M_{2,q}'$ & $xM_{0,q}'$ \\
\zeroblock & \zeroblock & \zeroblock \\
$M_{0,q}'$ & $M_{0,q}'$ & \zeroblock
\end{tabular}\right],
$
\end{minipage}}

\bigskip

\resizebox{\textwidth}{!}{
\begin{minipage}{1.2\textwidth}
$M_{2,(q+1)} =
\left[\begin{tabular}{R{1cm}R{1cm}R{1cm}}
$yM_{0,q}$ & $yM_{0,q}$ & \zeroblock \\
\zeroblock & \zeroblock & \zeroblock \\
\zeroblock & \zeroblock & \zeroblock
\end{tabular}\right],
$\text{ and }
$M_{2,(q+1)}' =
\left[\begin{tabular}{R{1cm}R{1cm}R{1cm}}
$yM_{0,q}'$ & $yM_{0,q}'$ & \zeroblock \\
\zeroblock & \zeroblock & \zeroblock \\
\zeroblock & \zeroblock & \zeroblock
\end{tabular}\right].
$
\end{minipage}}
\end{center}
\end{theorem}

\begin{proof}
The proof will be along the same lines as the proof of \cref{thm:generalPositions}, thus some of the details will be omitted.

The bar-state matrices count bar mosaics of length $q$. The matrix $M_{k,q}$ counts those bars with starting label $k$ and ending label 0, while the matrix $M_{k,q}'$ counts those with starting label $k$ and ending label 1.

The labels will be ternary strings, ordered lexicographically. Thus, the strings of length $2$ are ordered as
\[00,01,02,10,11,12,20,21,22.\]
The column labels of the matrix again correspond to the top label of the bar mosaic, while the row label is the bottom label. Each bar is represented by a monomial $x^ay^b$ where $a$ is the total number of $T_1$ and $T_2$ in the bar and $b$ is the total number of $T_7$ and $T_8$ in the bar.

Each matrix has $9$ blocks, which we label using roman numerals as shown in \cref{fig:9blocklabels}, where we also show the corresponding top and bottom labels for the leftmost tile.

\begin{figure}[ht]
\[\begin{array}{cc}
	\mbox{\diagbox[width=7em]{bottom}{top}}	& \begin{array}{ccc} 0 \ \ & \ \ \ 1 \ \ & \ \ 2\end{array}\\[2ex]
	\begin{array}{c} 0\\ 1\\ 2\\ \end{array}		&\left[\begin{array}{ccc}
													I & II & III\\
													IV & V & VI\\
													VII & IIX & IX\\	
												\end{array}\right]\\
\end{array}\]
\caption{The block construction of the state-matrices for maximal \domineering positions.}
\label{fig:9blocklabels}
\end{figure} 

In $M_{0,q+1}$, we have a left label of $0$. The only tile that meets the conditions for block I (and $M_{0,q+1}$) is $T_9$. As the right edge label of $T_9$ is $2$, the possible completions of our bar are given by $M_{2,q}$ and so our block I entry is $M_{2,q}$.

The conditions for block II force $T_{10}$ as the starting tile and the entry is again $M_{2,q}$.

The conditions for block III force $T_1$ --- giving the entry $xM_{0,q}$ as $T_1$ has right label of $0$. For block IV, tile $T_0$ is the only possible starting tile, giving entry $M_{1,q}$. For blocks V and VI we would need a tile with bottom label $1$ and top label $1$ or $2$, respectively, and no such tiles exist so our entry is $\zero$.

The conditions for block VII are top label of $0$ and bottom label of $2$; the only tile that works is $T_3$ and hence the entry is $M_{0,q}$. The conditions for block VIII are similar but with top label $1$ so $T_4$ is the only tile that works. 

The conditions for block IX are top label of $2$ and bottom label of $2$ which never occurs and thus our entry is $\zero$.

In $M_{1,q+1}$, we have a starting (left) label of $1$. This forces $T_{11}$ in block I; $T_{12}$ in block II; $T_2$ in block III; $T_5$ in block VII; and $T_6$ in block VIII. The remaining blocks have no possible tiles. Thus our entries for the matrix are as above.	

In $M_{2,q+1}$, we have a starting (left) label of $2$. Thus the only possible tiles are $T_7$ and $T_8$. In block I the only possible tile is $T_7$, giving $yM_{0,q}$. In block II the tile is $T_8$, giving $yM_{0,q}$. All other blocks have no possible tiles, so are $\zero$.

Note that the arguments for $M_{0,q+1}'$, $M_{1,q+1}'$, and $M_{2,q+1}'$ are similar as the possible left-most tile is identical and only the ending label changes, thus we skip these.

Now, restricting to tilings, we will only consider the bar mosaics of length $n$ with left label 0 and right labels 0 or 1, so $M_{0,n}+M_{0,n}'$. As in the general positions case, stacking the bar mosaics corresponds to matrix multiplication. Thus the mosaics with all left labels 0 and right labels 0 or 1 are enumerated in $(M_{0,n}+M_{0,n}')^m$. Finally, we need to restrict to top labels being all 0 and bottom labels 0 or 1. This means we need to sum the entries in the first column which are in rows numbered $\displaystyle 1+\sum_{i=1}^n u_i 3^i$  where the $u_i$ are all 0 or 1 (the ternary expansion of the row number$-1$ has no 2s), giving our result.
\end{proof}

\begin{example}
Using the recursion we find that
\[M_{0,2}=\left[\begin{array}{ccc:ccc:ccc}
y & y & 0 & y & y & 0 & 0 & 0 & x^2 \\
0 & 0 & 0 & 0 & 0 & 0 & 0 & 0 & 0 \\
0 & 0 & 0 & 0 & 0 & 0 & x & x & 0 \\
\hdashline
0 & 0 & x & 0 & 0 & 0 & 0 & 0 & 0 \\
0 & 0 & 0 & 0 & 0 & 0 & 0 & 0 & 0 \\
1 & 1 & 0 & 0 & 0 & 0 & 0 & 0 & 0 \\
\hdashline
0 & 0 & x & 0 & 0 & x & 0 & 0 & 0 \\
0 & 0 & 0 & 0 & 0 & 0 & 0 & 0 & 0 \\
1 & 1 & 0 & 1 & 1 & 0 & 0 & 0 & 0
\end{array}\right].\]

Taking the third power and summing the relevant entries, we get that
\[F_{3,2}(x,y)= 2x^2y + 2x^2 + y^3.\]
\end{example}

\begin{remark}
As for the general positions case, we can use these matrices to find the generating polynomial for non-rectangular boards. As the technique is the same, we do not demonstrate it here.
\end{remark}

\section{Counting Left and Right ends}\label{sec:ends}

Similar to maximal positions, we are also able to count Left or Right ends by choosing the tiles appropriately so that the adjacency condition alone forces such an end. We will just be counting the Right ends in this section as it results in bar-state matrices of size $2^q\times 2^q$, while the Left ends, using a similar argument, would need bar-state matrices of size $3^q\times 3^q$ (although fewer matrices). Alternatively, the generating polynomial for Left ends on an $m\times n$ board can be found using the generating polynomial of Right ends on an $n\times m$ board by switching $x$ and $y$. Further, if we would like to find the number of Left or Right ends, we add the number of positions of each type and then subtract the number of maximal positions as these are the positions that are Left and Right ends.

For Right ends, we may have two empty squares vertically adjacent as Left may potentially still play a domino. But we must not have two empty squares horizontally adjacent as Right must have no moves. To achieve this using the adjacency condition, we use the tiles  given in \cref{fig:end-tiles}.
\begin{figure}[!ht]
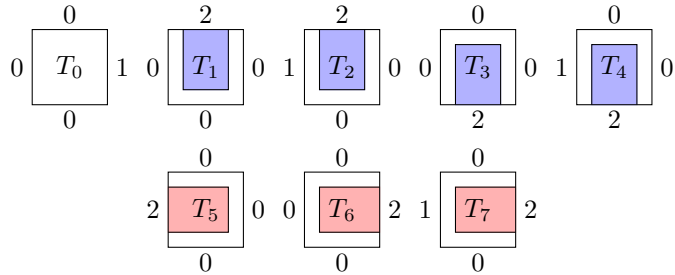

\begin{center}
\ReTO \ReTI \ReTII \ReTIII \ReTIV 

\ReTV \ReTVI \ReTVII
\end{center}
\caption{The tiles used to count Right ends in \domineering.}
\label{fig:end-tiles}
\end{figure}

The two conditions for a tiling to be a \domineering Right end are as follows:
\begin{enumerate}
\item \textit{Adjacency condition}: All shared edges of adjacent tiles have the same label.
\item \textit{Boundary condition}: The left, top, and bottom edges of the tiling have labels 0; the right edge has labels 0 or 1.
\end{enumerate}

\cref{fig:end-tiling} contains an example of a position and the corresponding tiling.
\begin{figure}
\resizebox{\textwidth}{!}{
\begin{tikzpicture}
\tikzchessboard[4]
\Ldomino{1}{1}
\Ldomino{1}{3}
\Rdomino{1}{0}
\Ldomino{3}{1}
\Rdomino{2}{3}
\Rdomino{2}{4}
\Ldomino{4}{0}

\begin{scope}[shift={(7,0)}]
\ReTOboard{0}{0}
\ReTOboard{0}{1}
\ReTOboard{0}{2}
\ReTOboard{0}{3}
\ReTOboard{0}{4}
\ReTOboard{2}{1}
\ReTOboard{2}{2}
\ReTOboard{3}{0}
\ReTOboard{4}{2}
\ReTOboard{4}{3}
\ReTOboard{4}{4}
\ReTIIboard{1}{1}
\ReTIIboard{1}{3}
\ReTIIboard{3}{1}
\ReTIIboard{4}{0}
\ReTIIIboard{4}{1}
\ReTIVboard{1}{2}
\ReTIVboard{1}{4}
\ReTIVboard{3}{2}
\ReTVboard{2}{0}
\ReTVboard{3}{3}
\ReTVboard{3}{4}
\ReTVIboard{2}{3}
\ReTVIboard{2}{4}
\ReTVIIboard{1}{0}
\end{scope}
\end{tikzpicture}}
\caption{A Right end on a $5\times 5$ board in \domineering and its equivalent tiling.}
\label{fig:end-tiling}
\end{figure}
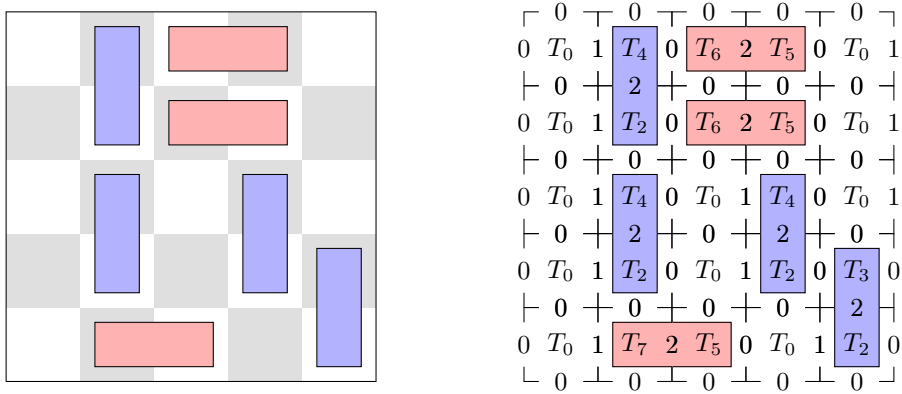

We use $RE$ for naming the bar-state matrices. The matrix $RE_{k,q}$ counts those bar mosaics of length $q$ with starting label $k$ and ending label $0$, while the matrix $RE_{k,q}'$ counts those with starting label $k$ and ending label $1$.

\begin{theorem}
The generating polynomial of \domineering Right ends on an $m\times n$ board is the (1,1) entry of $(RE_{0,n}+RE_{0,n}')^m$ where $RE_{0,0} = \begin{bmatrix} 1\\ \end{bmatrix}$,
$RE_{1,0} = \begin{bmatrix} 0\\ \end{bmatrix}$, $RE_{2,0} = \begin{bmatrix} 0\\ \end{bmatrix}$, $RE_{0,0}' = \begin{bmatrix} 0\\ \end{bmatrix}$,
$RE_{1,0}' = \begin{bmatrix} 1\\ \end{bmatrix}$, $RE_{2,0}' = \begin{bmatrix} 0\\ \end{bmatrix}$,
\begin{center}
\resizebox{\textwidth}{!}{
\begin{minipage}{1.1\textwidth}
$RE_{0,q+1} =
\left[\begin{tabular}{R{1.25cm}R{1.25cm}}
$RE_{1,q}$ $+RE_{2,q}$ & $xRE_{0,q}$ \\
$RE_{0,q}$ & \zeroblock
\end{tabular}\right],
$\hspace{0.75cm}
$RE_{0,q+1}' =
\left[\begin{tabular}{R{1.25cm}R{1.25cm}}
$RE_{1,q}'$ $+RE_{2,q}'$ & $xRE_{0,q}'$ \\
$RE_{0,q}'$ & \zeroblock
\end{tabular}\right],
$
\end{minipage}}

\bigskip
\resizebox{\textwidth}{!}{
\begin{minipage}{1.1\textwidth}
$RE_{1,q+1} =
\left[\begin{tabular}{R{1.25cm}R{1.25cm}}
$RE_{2,q}$ & $xRE_{0,q}$ \\
$RE_{0,q}$ & \zeroblock
\end{tabular}\right],
$\hspace{0.75cm}
$RE_{1,q+1}' =
\left[\begin{tabular}{R{1.25cm}R{1.25cm}}
$RE_{2,q}'$ & $xRE_{0,q}'$ \\
$RE_{0,q}'$ & \zeroblock
\end{tabular}\right],
$
\end{minipage}}

\bigskip
\resizebox{\textwidth}{!}{
\begin{minipage}{1.1\textwidth}
$RE_{2,q+1} =
\left[\begin{tabular}{R{1.25cm}R{1.25cm}}
$yRE_{0,q}$ & \zeroblock \\
\zeroblock & \zeroblock
\end{tabular}\right],
$\text{ and }
$RE_{2,q+1}' =
\left[\begin{tabular}{R{1.25cm}R{1.25cm}}
$yRE_{0,q}'$ & \zeroblock \\
\zeroblock & \zeroblock
\end{tabular}\right].
$
\end{minipage}}
\end{center}
\end{theorem}

We give no proof here as the argument is simply a combination of the arguments for the general and maximal positions.

In \cref{tab:number-of-ends} we give the total number of Right ends on an $m\times n$ board (the table for Left ends is the transpose of \cref{tab:number-of-ends}). \cref{tab:number-of-ends} is in the OEIS, read by antidiagonals, as \seqnum{A332862}.

\begin{table}[!ht]
\centering
\begin{tabular}{c||r|r|r|r|r}
\mbox{\diagbox[width=3em]{$m$}{$n$}} & 1 & 2 & 3 & 4 & 5\\\hline\hline
1 & 1 & 1 & 2 & 2 & 3 \\\hline 
2 & 2 & 4 & 11 & 25 & 61 \\\hline
3 & 3 & 9 & 48 & 172 & 731 \\\hline
4 & 5 & 25 & 227 & 1,427 & 10,388 \\\hline
5 & 8 & 64 & 1,054 & 11,134 & 140,555 \\\hline
6 & 13 & 169 & 4,921 & 88,733 & 1,932,067 \\\hline
7 & 21 & 441 & 22,944 & 701,926 & 26,425,981 \\\hline
8 & 34 & 1,156 & 107,017 & 5,567,467 & 362,036,629 
\end{tabular}
\vspace{1em}

\begin{tabular}{c||r|r|r}
\mbox{\diagbox[width=3em]{$m$}{$n$}} & 6 & 7 & 8\\\hline\hline
1 & 4 & 5 & 7\\\hline
2 & 146 & 351 & 844\\\hline
3 & 2,976 & 12,039 & 49,401\\\hline
4 & 72,751 & 510,779 & 3,604,887\\\hline
5 & 1,693,116 & 20,414,525 & 248,119,648\\\hline
6 & 40,008,789 & 831,347,033 & 17,385,222,733\\\hline
7 & 941,088,936 & 33,656,587,715 & 1,211,649,519,869\\\hline
8 & 22,168,654,178 & 1,365,206,879,940 & 84,588,476,099,284
\end{tabular}

\caption{The number of Right ends in \domineering on an $m\times n$ board.}
\label{tab:number-of-ends}
\end{table}

The number of Right ends on a $1\times n$ board is recursively given by $a_1=1$, $a_2=1$, $a_3=2$, $a_n=a_{n-2}+a_{n-3}$ (the Padovan sequence shifted, OEIS sequence \seqnum{A000931}). Consider the leftmost square, if the square contains a domino, then the strip to the right of the domino can contain any Right end of length $n-2$. If the square is empty, then a domino has to be immediately to the right of it, and the remaining strip of length $n-3$ can contain any Right end.

In the positions of shape $m\times 1$ Right cannot play, thus all positions are Right ends. The sequence of these numbers is the Fibonacci sequence (OEIS sequence \seqnum{A000045}) as every position is a combination of empty squares and Left dominoes.

It appears that the number of Right ends on an $m\times 2$ board is given by the squared Fibonacci numbers (OEIS sequence \seqnum{A007598}) and the number of positions on an $m\times 3$ is given by the recursion $a_m=4a_{m-1}+4a_{m-2}-4a_{m-3}-a_{m-4}$ (OEIS sequence \seqnum{A054894}).

The main diagonal of \cref{tab:number-of-ends}, the number of Right ends on an $n \times n$ board, is in the OEIS as \seqnum{A332865}. Other sequences for the number of Right ends do not appear in the OEIS at this point.

\section{Further work} \label{sec:FurtherWork}

We suggest three broad avenues for continuing the work in this paper. One is to explore our original motivation and connect the enumeration of \domineering to questions of the algebraic structure of placement game positions. Secondly, we expect the methods demonstrated in this paper to be useful for other placement games played on grid-like boards (such as Cartesian products of cycles and paths). Lastly, there is the possibility of counting positions to analyze particular game situations either in theory or practice to develop better play strategies.

\section{Acknowledgments}

The authors thank Mount Allison University in Sackville NB, Canada. This research project started while both authors were members of the Department of Mathematics and Computer Science.

The first author's research was supported in part by the Natural Sciences and Engineering Research Council of Canada (funding reference number PDF-516619-2018).

\bibliographystyle{plain}

\bigskip
\hrule
\bigskip

\noindent 2010 {\it Mathematics Subject Classification}:
Primary 91A46; Secondary 05C30 and 05C57. 

\noindent \emph{Keywords:} combinatorial game, \Domineering, enumeration, matchings, grid graph.

\bigskip
\hrule
\bigskip

\end{document}